\documentclass[12pt,reqno,intlimits,twoside]{amsart}
\usepackage[cp1251]{inputenc}
\usepackage[T2A]{fontenc}
\usepackage[english]{babel}
\inputencoding{cp1251}
\usepackage{amsfonts,amsmath,amsxtra,amsthm,amssymb,latexsym,euscript}
\usepackage{graphicx}

\newtheorem{thm}{Theorem}[section]
\newtheorem{cor}{Corollary}[section]
\newtheorem{lem}{Lemma}[section]
\newtheorem{prop}{Proposition}[section]
\theoremstyle{definition}
\newtheorem{defn}{Definition}[section]
\theoremstyle{remark}

\newtheorem{eg}{Example}[section]

\DeclareMathOperator*{\dom}{\mathrm{dom}}
\let\im\Im

\DeclareMathOperator*{\PLUS}{\oplus} \allowdisplaybreaks
\begin{document}

\allowdisplaybreaks
\title[Isospectrality for graph Laplacians]{Isospectrality for graph Laplacians under the change of coupling at graph vertices}
\author{Yulia Ershova}
\address{Institute of Mathematics, National Academy of Sciences of Ukraine. 01601 Ukraine, Kiev-4,
3, Tereschenkivska st.} \email{julija.ershova@gmail.com}
\author{Irina I. Karpenko}
\address{Department of Algebra and Functional Analysis, V.I. Vernadsky Taurida National university. 95007 Simferopol, 4 Vernadsky pr.}
\email{i\_karpenko@ukr.net}
\author{Alexander V. Kiselev}
\address{Department of Functional Analysis, Pidstryhach Institute for Applied Problems of Mechanics and Mathematics,
National Academy of Sciences of Ukraine, 3-b Naukova Str. 79060,
L'viv, Ukraine}
\address{Department of Higher Mathematics and Mathematical Physics,
St. Petersburg State University, 1 Ulianovskaya Street, St.
Petersburg, St. Peterhoff 198504 Russia}
\email{alexander.v.kiselev@gmail.com}
\thanks{The third authors' work was partially supported by the RFBR, grant no. 12-01-00215-a.}


\begin{abstract}
Laplacian operators on finite compact metric graphs are considered
under the assumption that matching conditions at graph vertices
are of $\delta$ and $\delta'$ types. An infinite set of trace
formulae is obtained which link together two different quantum
graphs under the assumption that their spectra coincide. The
general case of graph Schr\"odinger operators is also considered,
yielding the first trace formula. Tightness of results obtained
under no additional restrictions on edge lengths is demonstrated
by an example. Further examples are scrutinized when edge lengths
are assumed to be rationally independent. In all but one of these
impossibility of isospectral configurations is ascertained.
\end{abstract}
\maketitle

\section{Introduction}
In the present paper we focus our attention on the so-called
quantum graph, i.e., a metric graph $\Gamma$ with an associated
second-order differential operator acting in Hilbert space
$L^2(\Gamma)$ of square summable functions with an additional
assumption that functions belonging to the domain of the operator
are coupled by certain matching conditions at  graph vertices.
Recently these operators have attracted a considerable interest of
both physicists and mathematicians due to a number of important
physical applications. Extensive literature on the subject is
surveyed in, e.g., \cite{Kuchment,Kuchment2}.

The present paper is devoted to the study of an inverse spectral
problem for Laplace and Schr\"odinger operators on finite compact
metric graphs. One might classify the possible inverse problems on
quantum graphs in the following way: (i) given spectral data, edge
potentials and matching conditions, to reconstruct the metric
graph; (ii) given the metric graph, edge potentials and spectral
data, to reconstruct matching conditions; (iii) given the metric
graph, spectral data and matching conditions, to reconstruct edge
potentials.

There exists an extensive literature devoted to the problem
\emph{(i)}. To name just a few, we would like to mention
pioneering works \cite{Roth,Smil1,Smil2} and later contributions
\cite{Kura1,Kura2,Kura3,Kostrykin}. Different approaches to the
same problem were developed, e.g., in
\cite{pivo,Belishev_tree,Belishev_cycles}, see also
\cite{NabokoKurasov} for the analysis of a related problem. The
problem \emph{(iii)} is a generalization of the classical inverse
problem for Sturm-Liouville operators and thus unsurprisingly has
attracted by far the most interest; it is nonetheless beyond the
scope of the present paper.

The problem \emph{(ii)} has attracted much less attention. We
believe it was first treated in \cite{Carlson} for the square of
self-adjoint operator of the first derivative on a graph. Then,
after being mentioned in \cite{Kura2,Kura3}, it was treated in
\cite{Avdonin}, but only in the case of star graphs. In our papers
\cite{Yorzh1,Yorzh2} we suggested an approach based on the theory
of boundary triples, leading to the asymptotic analysis of
Weyl-Titchmarsh M-function of the graph.

Unlike \cite{Carlson,Avdonin}, in the present paper we consider
the case of a general connected compact finite metric graph (in
particular, this graph is allowed to possess cycles and loops),
but only for two possible classes of matching conditions at graph
vertices. Namely, each vertex is allowed to have matching of
either $\delta$ or $\delta'$ type (see Section 2 for definitions).
The named two classes singled out by us prove to be physically
viable \cite{Exner1, Exner2}.

\section{Preliminaries}

\subsection{Definition of the Laplace operator on a quantum
graph}\label{sect_def_laplace}

We call $\Gamma=\Gamma(\mathbf{E_\Gamma},\sigma)$ a finite compact
metric graph, if it is a collection of a finite non-empty set
$\mathbf{E_\Gamma}$ of compact intervals
$\Delta_j=[x_{2j-1},x_{2j}]$, $j=1,2,\ldots, n$, called
\emph{edges}, and of a partition $\sigma$ of the set of endpoints
$\{x_k\}_{k=1}^{2n}$ into $N$ classes,
$\mathbf{V_\Gamma}=\bigcup^N_{m=1} V_m$. The equivalence classes
 $V_m$, $m=1,2,\ldots,N$ will be called \emph{vertices} and the number of elements belonging to the set $V_m$ will be called the
 \emph{valence} (or, alternatively, \emph{degree}) of the vertex
$V_m$ (denoted $\deg V_m\equiv\gamma_m$).

With a finite compact metric graph $\Gamma$ we associate Hilbert
spaces $L_2(\Gamma)=\PLUS_{j=1}^n L_2(\Delta_j)$  and
$W_2^2(\Gamma)=\oplus_{j=1}^n W_2^2(\Delta_j).$ These spaces
obviously do not feel the graph connectivity, being the same for
each graph with the same number of edges of same lengths.

In what follows, we single out two natural \cite{Exner1} classes
of so-called \emph{matching conditions} which lead to a properly
defined self-adjoint operator on the graph $\Gamma$, namely,
matching conditions of $\delta$ and $\delta'$ types. In order to
describe these, we introduce the following notation. For a
function $f\in W_2^2(\Gamma)$, we will use throughout the
following definition of the normal derivative $\partial_nf(x_{j})$
at the endpoints of the interval $\Delta_k:$
\begin{equation*}\label{fdnfx}
\partial_n f(x_j)=\left\{ \begin{array}{rl} f'(x_j),&\mbox{ if } x_j \mbox{ is the left endpoint of the edge},\\
-f'(x_j),&\mbox{ if } x_j \mbox{ is the right endpoint of the
edge.}
\end{array}\right.
\end{equation*}

It will be convenient to introduce the following notation for a
function $f\in W_2^2(\Gamma)$ at any graph vertex:
\begin{equation*}
    f^\Sigma(V_k)=\sum\limits_{x_j\in V_k}f(x_j), \quad \partial_n^\Sigma f(V_k)=\sum\limits_{x_j\in V_k}\partial_n
    f(x_j).
\end{equation*}

Associate either of the two symbols, $\delta$ or $\delta'$, to
each vertex of the graph $\Gamma$. The graph thus obtained will be
referred to as \textit{marked} and denoted $\Gamma_{\delta}$. Any
marked graph  $\Gamma_{\delta}$ determines the lineal
$$
\mathcal{D}(\Gamma_{\delta}):=\left\{f\in W_2^2(\Gamma)\ \left|\
                                                 \begin{array}{l}
                                                   f\ \mbox{is continuous at all}\\ \mbox{internal vertices of }\ \delta \mbox{ type}, \\
                                                   \partial_nf(x_i)=\partial_nf(x_j) \forall i,j: x_i,x_j\in V \mbox{ at all}\\ \mbox{internal vertices $V$ of }\  \delta' \mbox{ type} \\
                                                 \end{array}                                          \right. \right\}.
$$
Note that the latter definition imposes no restrictions on the
functions from $\mathcal{D}(\Gamma_{\delta})$ at boundary vertices
of the graph, i.e., at vertices of valence 1. For reasons of
convenience, we refer to all graph vertices of higher valence as
\emph{internal vertices} throughout.

We remark that if the vertex $V_k$ of valence $\gamma_k$ is of
$\delta$ type, then obviously $f^\Sigma(V_k)=\gamma_kf(x_j),\ x_j
\in V_k.$ In the same way, for a vertex  $V_k$ of $\delta'$ type
one has $\partial_n^\Sigma f(V_k)=\gamma_k\partial_n f(x_j),\ x_j
\in V_k$. In fact, throughout the rest of the paper we will only
use the notation $f^\Sigma(V_k)$ and $\partial_n^\Sigma (V_k)$ in
these two respective cases.

In Hilbert space $L_2(\Gamma)$ consider the operator $A_{\min}$,
defined on each edge of the graph by the differential expression
$-\frac{d^2}{dx^2}$, the domain of which
$\mathrm{dom}(A_{\mathrm{min}})$ consists of all functions  $f\in
\mathcal{D}(\Gamma_{\delta})$ such that
\begin{align}\label{Eq_domAmin}
&  f^\Sigma(V_k)=0,\ \partial_n^\Sigma f(V_k)=0\ \forall k.
\end{align}

Obviously, $A_{\min}$ is a closed symmetric operator, which will
be henceforth referred to as \textit{the symmetric operator of the
graph} $\Gamma_{\delta}$. The adjoint to it $A_{\max}:=A^*_{\min}$
is defined by the same differential expression on the domain
$\mathcal{D}(\Gamma_{\delta})$. The deficiency indices of
$A_{\min}$ are equal to $(N,N)$.

We are now ready to define the Laplacian $A_{\vec\alpha}$ on the
graph $\Gamma_{\delta}$ which is an operator of the negative
second derivative on functions from
$f\in\mathcal{D}(\Gamma_{\delta})$ subject to the following
additional \emph{matching conditions}.
\begin{itemize}
\item[\textbf{($\delta$)}] If $V_k$ is of $\delta$ type, then
$$
\sum_{x_j \in V_k} \partial _n f(x_j)=\frac{\alpha_k}{\gamma_k}
f^\Sigma(V_k).
$$
\item[\textbf{($\delta'$)}] If $V_k$ is of $\delta'$ type, then
$$
\sum_{x_j \in V_k}
f(x_j)=-\frac{\alpha_k}{\gamma_k}\partial_n^\Sigma f(V_k).
$$
\end{itemize}
Here $\vec\alpha=(\alpha_1,\alpha_2,...\alpha_N)$ is a set of
arbitrary real constants which we will refer to as \emph{coupling
constants}.

Note that matching conditions at internal vertices reflect the
graph connectivity: if two graphs having the same set of edges
have different topology, the resulting operators are different.

Provided that all coupling constants $\alpha_m$, $m=1\dots N$, are
real, it is easy to ascertain that the operator $A_{\vec{\alpha}}$
is a proper self-adjoint extension of the operator $A_{\min}$ in
Hilbert space $L_2(\Gamma)$ \cite{Exner1,KostrykinS}.

Clearly, the self-adjoint operator thus defined on a finite
compact metric graph has purely discrete spectrum that accumulates
to $+\infty$.

Note that w.l.o.g. each edge $\Delta_j$ of the graph $\Gamma$ can
be considered to be an interval $[0,l_j]$, where
$l_j=x_{2j}-x_{2j-1}$, $j=1,\dots, n$ is the length of the
corresponding edge. Throughout the present paper we will therefore
only consider this situation.

In order to treat the inverse spectral problem (ii) for graph
Laplacians, we will first need to get an explicit expression for
 the generalized Weyl-Titchmarsh M-function of the operator considered. The most elegant and straightforward way of doing so is in our view
by utilizing the apparatus of boundary triples developed in
\cite{Gor,Ko1,Koch,DM}. We briefly recall what is essential for
our work.

\subsection{Boundary triples and the Weyl-Titchmarsh matrix
M-function}

Suppose that $A_{\min}$ is a symmetric densely defined closed
linear operator acting in Hilbert space $H$. Assume that
$A_{\min}$ is completely non-self-adjoint (simple)\footnote{The
condition of simplicity of $A_{\min}$ was studied in
\cite{Karpenko}, where necessary and sufficient conditions of this
property were obtained. In the simplest case, provided that all
edge lengths are rationally independent and the graph contains no
loops, simplicity is guaranteed. For the problems discussed in the
present paper, simplicity is in fact irrelevant. Indeed, even if
$A_{\min}$ is not simple, under the condition of isospectrality of
two graph Laplacians the formula \eqref{Eq_F/tildeFb} of Section 4
still holds, see below for details, which is sufficient for our
analysis, although in this case not all of the spectrum of a graph
Laplacian is ``visible'' to both factors of
\eqref{Eq_F/tildeFb}.}, i.e., there exists no reducing subspace
$H_0$ in $H$ such that the restriction $A_{\min}|H_0$ is a
selfadjoint operator in $H_0.$ Further assume that the deficiency
indices of $A_{\min}$ (possibly being infinite) are equal:
$n_+(A_{\min})=n_-(A_{\min})\le\infty.$

\begin{defn}[\cite{Gor,Ko1,DM}]\label{Def_BoundTrip}
Let $\Gamma_0,\ \Gamma_1$ be linear mappings of $\dom(A_{\max})$
to $\mathcal{H}$ -- a separable Hilbert space. The triple
$(\mathcal{H}, \Gamma_0,\Gamma_1)$ is called \emph{a boundary
triple} for the operator $A_{\max}$ if:
\begin{enumerate}
\item for all $f,g\in \dom(A_{\max})$
$$
(A_{\max} f,g)_H -(f, A_{\max} g)_H = (\Gamma_1 f, \Gamma_0
g)_{\mathcal{H}}-(\Gamma_0 f, \Gamma_1 g)_{\mathcal{H}}.
$$
\item the mapping $\gamma$ defined as $f\longmapsto (\Gamma_0 f;
\Gamma_1 f),$ $f\in \dom(A_{\max})$ is surjective, i.e., for all
$Y_0,Y_1\in\mathcal{H}$ there exists an element $y\in
\dom(A_{\max})$ such that $\Gamma_0 y=Y_0,\ \Gamma_1 y =Y_1.$
\end{enumerate}
\end{defn}

A boundary triple can be constructed for any operator $A_{\max}$
of the class considered. Moreover, the space $\mathcal H$ can be
chosen in a way such that $\dim \mathcal H=n_+=n_-.$ In particular
one has
$A_{\min}=A_{\max}|_{\left(\ker\Gamma_{0}\cap\ker\Gamma_{1}\right)}$.

We further single out two proper self-adjoint extensions of the
operator $A_{\min}$: $ A_{\infty}:=A_{\max}{|}\ker\Gamma_{0}$,
$A_{0}:=A_{\max}{|}\ker\Gamma_{1}. $

A nontrivial extension ${A}_B$ of the operator $A_{\min}$ such
that $A_{\min}\subset  A_B\subset A_{\max}$  is called
\emph{almost solvable} if for every $f\in \dom(A_{\max})$
$$
f\in \dom({A_B})\text{ if and only if } \Gamma_1 f=B\Gamma_0 f.
$$
for a bounded in $\mathcal H$ operator $B$.

The generalized Weyl-Titchmarsh M-function is then defined as
follows.
\begin{defn}[\cite{DM,Gor,Koch}]\label{DefWeylFunc}
Let $A_{\min}$ be a closed densely defined symmetric operator,
$n_+(A_{\min})=n_-(A_{\min}),$ $(\mathcal{H},\Gamma_0,\Gamma_1)$
being its boundary triple. The operator-function $M(\lambda),$
defined by
\begin{equation*}\label{Eq_Func_Weyl}
M(\lambda)\Gamma_0 f_{\lambda}=\Gamma_1 f_{\lambda}, \
f_{\lambda}\in \ker (A_{\max}-\lambda),\  \lambda\in
\mathbb{C}_\pm,
\end{equation*}
is called the Weyl-Titchmarsh M-function of a symmetric operator
$A_{\min}.$
\end{defn}

The following Theorem describing properties of the M-function
clarifies its meaning.
\begin{thm}[\cite{Gor,DM}, in the form adopted by \cite{RyzhovOTAA}]\label{Th_Spectra}
Let $M(\lambda)$ be the M-function of a symmetric operator
$A_{\min}$ with equal deficiency indices
($n_+(A_{\min})=n_-(A_{\min})<\infty$). Let $A_B$ be an almost
solvable extension of $A_{\min}$ corresponding to a bounded
operator $B.$ Then for every $\lambda\in \mathbb{C}:$
\begin{enumerate}
\item $M(\lambda)$ is an analytic operator-function when
$Im\lambda\not=0$, its values being bounded linear operators in
$\mathcal{H}.$ \item $(Im\ M(\lambda))Im\ \lambda>0$ when $Im
\lambda\not =0.$ \item $M(\lambda)^*=M(\overline{\lambda})$ when
$Im \lambda\not =0.$ \item $\lambda_0\in \rho(A_B)$ if and only if
$(B-M(\lambda))^{-1}$ admits analytic continuation into the point
$\lambda_0$.
\end{enumerate}
\end{thm}

\section{Weyl-Titchmarsh function for the graph Laplacian}
In this Section we derive an explicit formula for Weyl-Titchmarsh
M-function pertaining to graph Laplacians of the class considered.

\begin{thm}\label{Th_Bound_Triple}
Let $A_{\min}$ be the symmetric operator of the graph
$\Gamma_{\delta}$ defined in \eqref{Eq_domAmin}. Then a boundary
triple for the operator $A_{\max}$ such that any proper
self-adjoint extension $A_{\vec\alpha}$ of Section
\ref{sect_def_laplace} is almost solvable can be defined as
follows: $\mathcal{H}=\mathbb{C}^N,$
\begin{equation}\label{Eq_Gamma_0}
     (\Gamma_0f)_k:=\frac{1}{\gamma_k}
     \begin{cases}
        f^\Sigma(V_k),
            &\ \mbox{if}\  V_k \mbox{ is a vertex of}\ \delta\mbox{ type},\\
        \partial_n^\Sigma f(V_k),
            &\ \mbox{if}\  V_k  \mbox{ is a vertex of}\
            \delta'\mbox{ type},
     \end{cases}
\end{equation}
\begin{equation}\label{Eq_Gamma_1}
     (\Gamma_1f)_k:=
     \begin{cases}
        \partial_n^\Sigma f(V_k),
            &\ \mbox{if}\  V_k  \mbox{ is a vertex of}\ \delta\mbox{ type},\\
        -f^\Sigma(V_k),
            &\ \mbox{if}\  V_k  \mbox{ is a vertex of}\
            \delta'\mbox{ type},
            \end{cases}
\end{equation}
where $\gamma_k=\mathrm{deg}~V_k$.

In this setting, the matrix $B$ parameterizing the almost solvable
extension  $A_{\vec\alpha}$ is diagonal,
$B=\mathrm{diag}\{\alpha_1,\alpha_2,\dots,\alpha_N\}$, where
$\{\alpha_k\}_{k=1}^N$ are coupling constants of conditions
$(\delta)$ and $(\delta')$.
\end{thm}

\begin{proof}
It suffices to observe that
\begin{equation*}\label{Eq_Delta}
          (A_{\max}f,g)-(f,A_{\max}g)=\sum\limits_{k=1}^N\frac{1}{\gamma_k}\Bigl(\partial_n^\Sigma f(V_{k})\overline{g^\Sigma(V_{k})}-
          f^\Sigma(V_{k})\overline{\partial_n^\Sigma g(V_{k})}\Bigr).
\end{equation*}
Then formulae~\eqref{Eq_Gamma_0},~\eqref{Eq_Gamma_1} and
conditions  $(\delta)$ and $(\delta')$, immediately imply the
claimed form of the matrix $B$.
\end{proof}

\begin{lem}\label{Lem_Delta_Vertice}
Assume that $V_k,\ V_j$ are two adjacent vertices of the graph
$\Gamma_{\delta}$ connected by an edge $e_{kj}$ of length
$l_{kj}$, $\gamma_k=\mathrm{deg}~V_k$. Let the function $f$ be in
$\ker (A_{\max}-\lambda)$ and put $f_{kj}:=f| e_{kj}$.

\textbf{(i)} If $V_k,\ V_j$ are both of $\delta$ type and
$f^\Sigma(V_k)=\gamma_k,\ f^\Sigma(V_j)=0,$ then
\begin{equation}\label{Eq_Delta_End}
\partial_nf_{kj}(V_k)=-\mu\cot\mu l_{kj},\ \partial_nf_{kj}(V_j)=\frac{\mu}{\sin\mu l_{kj}}
\end{equation}

\textbf{(ii)} If $V_k,\ V_j$ are two vertices of $\delta$ and
$\delta'$ types, respectively, and  $f^\Sigma(V_k)=\gamma_k,\
\partial_n^\Sigma f(V_j)=0,$ then
\begin{equation}\label{Eq_Delta'_End}
\partial_nf_{kj}(V_k)=\mu\tan\mu l_{kj},\ f_{kj}(V_j)=\frac{1}{\cos\mu l_{kj}}.
\end{equation}

\textbf{(iii)} If $V_k,\ V_j$ are two vertices of $\delta'$ and
$\delta$ types, respectively, and $\partial_n^\Sigma
f(V_k)=\gamma_k,\ f^\Sigma(V_j)=0,$ then
\begin{equation}\label{Eq_Delta_End_Delta'}
f_{kj}(V_k)=-\frac{1}{\mu}\tan\mu l_{kj},\ \partial_nf_{kj}(V_j)=-\frac{1}{\cos\mu l_{kj}}.
\end{equation}

\textbf{(iv)} If $V_k,\ V_j$ are both of $\delta'$ type and
$\partial_n^\Sigma f(V_k)=1,\
\partial_n^\Sigma f(V_j)=0,$ then
\begin{equation}\label{Eq_Delta'_End_Delta'}
f_{kj}(V_k)=\frac{1}{\mu}\cot\mu l_{kj},\
f_{kj}(V_j)=\frac{1}{\mu\sin\mu l_{kj}}.
\end{equation}
Here $\mu=\sqrt{\lambda}$ with the branch so chosen that $\Im
\mu\geq 0$.
\end{lem}

The \emph{proof} is an explicit computation.

\begin{thm}\label{Th_Weyl_Func}
Let $\Gamma_{\delta}$ be a marked compact metric graph and let the
operator $A_{\min}$ be the symmetric operator \eqref{Eq_domAmin}
of the graph $\Gamma_{\delta}$. Assume that the boundary triple
for $A_{\max}$ is $(\mathbb{C}^N,\ \Gamma_0,\ \Gamma_1)$, where
$N$ is the number of vertices in  $\Gamma$, whereas the operators
$\Gamma_0$ and $\Gamma_1$ are defined by
~\eqref{Eq_Gamma_0},~\eqref{Eq_Gamma_1}. Then the generalized
Weyl-Titchmarsh $M$-function is a $N\times N$ matrix with matrix
elements given by the following formula for a vertex  $V_k$ of
$\delta$ type: $m_{jk}(\lambda)=$
\begin{equation}\label{Eq_Weyl_Func_Delta}
\begin{cases}\scriptsize
-\mu\Bigl(\sum_{\Delta_t\in E_k}\cot\mu l_t-\sum_{\Delta_t\in
E'_k}\tan\mu l_t-    \\
-2\sum_{\Delta_t\in L_k}\tan\frac{\mu l_t}{2}\Bigr),
            & j=k,\\
\mu\sum_{\Delta_t\in C_{kj}}\frac{1}{\sin\mu l_t},
            & j\not=k,\  V_j \mbox{ is a vertex of }\\
            & \delta\mbox{ type adjacent to}\ V_k,\\
-\sum_{\Delta_t\in C'_{kj}}\frac{1}{\cos\mu l_t},
            & j\not=k,\  V_j \mbox{ is a vertex of }\\
            & \delta'\mbox{ type adjacent to }\ V_k,\\
            0,
            & j\not=k,\  V_j \mbox{ is a vertex }\\
            & \mbox{ not adjacent to }\ V_k,\\
     \end{cases}
\end{equation}
and by the following formula for a vertex $V_k$ of $\delta'$ type:
$m_{jk}(\lambda)=$
\begin{equation}\label{Eq_Weyl_Func_Delta'}
\begin{cases}
-\frac{1}{\mu}\Bigl(\sum_{\Delta_t\in E_k}\cot\mu
l_t-\sum_{\Delta_t\in E'_k}\tan\mu
l_t+\\
+2\sum_{\Delta_t\in L_k}\cot\frac{\mu
 l_t}{2}\Bigr),  & j=k,\\
-\sum_{\Delta_t\in C'_{kj}}\frac{1}{\cos\mu l_t},
            & j\not=k,\  V_j \mbox{ is a vertex of }\\
            & \delta\mbox{ type adjacent to }\ V_k,\\
-\frac{1}{\mu}\sum_{\Delta_t\in C_{kj}}\frac{1}{\sin\mu l_t},
            & j\not=k,\  V_j  \mbox{ is a vertex of }\\
            & \delta'\mbox{ type adjacent to }\ V_k,\\
            0,
            & j\not=k,\  V_j \mbox{ is a vertex}\\
            & \mbox{not adjacent to }\ V_k.\\
     \end{cases}
\end{equation}
Here $\mu=\sqrt{\lambda}$ (the branch such that $\im\mu\geq 0$),
$l_t$ is the length of the edge $\Delta_t$, $L_k$ is the set of
loops at the vertex $V_k,$ $E_k$ is the set of graph edges
incident to the vertex $V_k$ with both endpoints of the same type,
$E'_k$ is the set of graph edges incident to the vertex $V_k$ with
endpoints of different type, $C_{kj}$ is the set of graph edges
connecting vertices $V_k$ and $V_j$ of the same type, and finally,
$C'_{kj}$ is the set of graph edges connecting vertices  $V_k$ and
$V_j$ of different types.
\end{thm}

\begin{proof}
The proof is an explicit computation. Consider a function
$f_{\lambda}\in\ker(A_{\max}-\lambda)$. Let
$\Gamma_0f_{\lambda}=e_k,$ where $e_k=(0,0,\dots,{1},0,\dots,0)^T$
with $1$ in the $k$th position.

From Lemma \ref{Lem_Delta_Vertice} one gets the explicit
description of $f_\lambda$; it is then possible to compute
$\Gamma_1f_{\lambda}$ which yields the $k$th column of the
M-matrix sought provided that the graph has no loops.

If $\Gamma_\delta$ contains loops, it is easy to see that these
lead to contributions in diagonal entries as claimed.
\end{proof}

\begin{eg}\label{Example_mexican_dupa} Consider the following graph $\Gamma_{\delta}$:
\begin{center}
\includegraphics[width=.97\textwidth]{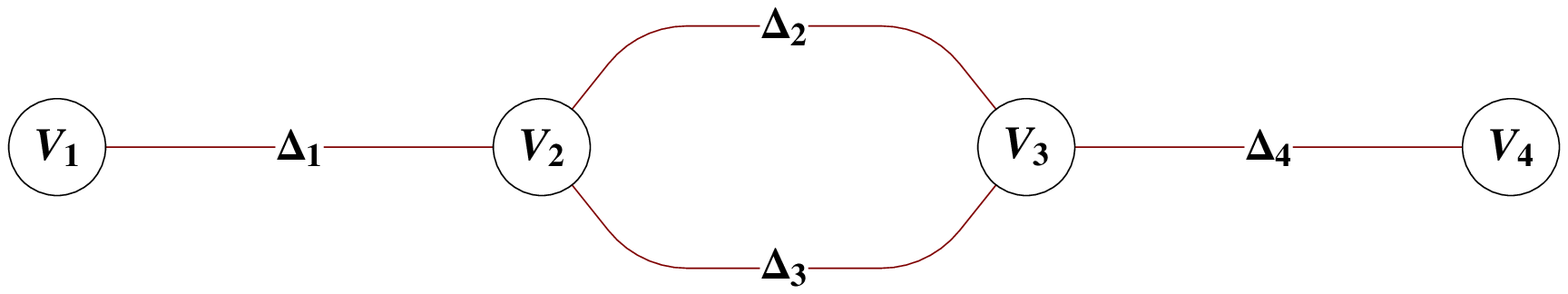}
\end{center}
Let the vertices $V_1$, $V_2$ be of $\delta$ type and the vertices
$V_3$, $V_4$ -- of  $\delta'$ type. Then within the setup of
Theorem~\ref{Th_Weyl_Func} the Weyl-Titchmarsh function admits the
form $M(\lambda)=$
\begin{SMALL}
\begin{equation*}
 \left(
      \begin{array}{cccc}
 -\mu\cot\mu l_1&\mu\csc\mu l_1  &    0    &  0 \\
 \mu\csc\mu l_1  &-\mu \left(\cot\mu l_1-\sum_{t=2}^3\tan\mu l_t\right)&-\sum_{j=2}^3\sec\mu l_t &  0 \\
 0 &-\sum_{j=2}^3\sec\mu l_t &-\frac{1}{\mu}\left(\cot\mu l_4-\sum_{t=2}^3\tan\mu l_t\right) &-\frac{1}{\mu}\csc\mu l_4\\
 0 & 0 & -\frac{1}{\mu}\csc\mu l_4             & -\frac{1}{\mu}\cot\mu l_4
      \end{array}
    \right)
\end{equation*}\end{SMALL}
\end{eg}

\section{Trace formulae for isospectral graph Laplacians}
In the present section, we apply the mathematical apparatus
developed in Section 3 in order to study isospectral (i.e., having
the same spectrum, counting multiplicities) quantum Laplacians
defined on a finite compact metric graph $\Gamma_\delta$.

Considering a pair of such Laplacians which differ  only in
coupling constants defining matching conditions we will derive an
infinite series of trace formulae (cf. \cite{Yorzh1} for an
analogous, although written in a different form, result in the
case of all vertices being of $\delta$ type).

\begin{thm}\label{Th_Trace_Form_Com}
Let $\Gamma_{\delta}$ be a finite compact metric graph with $N$
vertices in which $V_1,V_2,...,V_{N_1}$ are of $\delta$ type,
whereas $V_{N_1+1},V_{N_1+2},...,V_{N}$ are of $\delta'$ type. Let
$A_{\vec\alpha},\ A_{\vec{\widetilde{\alpha}}}$ be two graph
Laplacians on $\Gamma_{\delta}$ parameterized by coupling
constants $\{\alpha_{k}\}$ and $\{\widetilde{\alpha}_{k}\}$,
$k=\overline{1,N}$, respectively. If (point) spectra of the
operators  $A_{\vec\alpha}$ and $A_{\vec{\widetilde{\alpha}}}$
coincide counting multiplicities, the following infinite set of
trace formulae holds.
\begin{equation}\label{Eq_Trace_Form_Com}
\sum_{i=1}^{N_1}\frac{(-\alpha_i)^m}{\gamma_i^m}+\sum_{i=N_1+1,\
\alpha_i\not=0}^ {N}\frac{\gamma_{i}^m}{\alpha_{i}^m}=
\sum_{i=1}^{N_1}\frac{{(-\widetilde\alpha}_i)^m}{\gamma_i^m}+\sum_{i=N_1+1,\
\widetilde{\alpha}_i\not=0}^ {N}
\frac{\gamma_{i}^m}{\widetilde{\alpha}_{i}^m},
\end{equation}
$m=1,2,\dots$, where as above $\gamma_i=\deg V_i$. Moreover, the
sets $\{\alpha_{N_1+1},\dots,\alpha_N\}$ and
$\{\tilde\alpha_{N_1+1},\dots,\tilde\alpha_N\}$ have equal numbers
of zero elements.
\end{thm}
\begin{proof}
Using results of \cite[Chapter I]{Levin}, see \cite{Yorzh1} for
details, it is easy to ascertain that the fraction
$\det(B-M(\lambda))/\det(\widetilde B-M(\lambda))$ is an entire
function of exponential type of order not greater than $1/2$ under
the assumption that the spectra of $A_{\vec{\alpha}}$ and
$A_{\vec{\widetilde{\alpha}}}$ coincide (recall, that
$B=\mathrm{diag}\{\alpha_1,\dots,\alpha_N\}$ and $\tilde
B=\mathrm{diag}\{\tilde\alpha_1,\dots,\tilde\alpha_N\}$).
Moreover, it has no finite zeroes. Therefore by Hadamard theorem
\cite{Levin} one has
\begin{equation}\label{Eq_F/tildeFb}
\frac{\det(B-M(\lambda))}{\det(\widetilde B-M(\lambda))}=\exp(a)
\end{equation}
for some constant $a$.

Consider asymptotic expansions of the functions
$\det(B-M(\lambda))$ and  $\det(\widetilde{B}-M(\lambda))$ as
$\lambda\rightarrow -\infty$ along the real line. Using the
asymptotic expansion for $M(\lambda)$ provided by Theorem
\ref{Th_Weyl_Func} one has
\begin{gather*}
\det(B-M(\lambda))=\prod_{i=1}^{N_1}(\alpha_i+\gamma_i\tau)\prod_{i=N_1+1}^{N}(\alpha_i-\frac{\gamma_i}{\tau})+{o}(\tau^{-M})\\
\det(\widetilde{B}-M(\lambda))=\prod_{i=1}^{N_1}(\widetilde{\alpha}_i+\gamma_i\tau)\prod_{i=N_1+1}^{N}(\widetilde{\alpha}_i-
\frac{\gamma_i}{\tau})+{o}(\tau^{-M})
\end{gather*}
for any natural $M>0$, where $\tau=-i \sqrt{\lambda}\to +\infty.$
Therefore, passing to the limit in~\eqref{Eq_F/tildeFb} as
$\tau\rightarrow +\infty,$ one either gets
$$
\exp(a)=\frac{\prod_{i=N_1+1,\ \alpha_i\not=0}^{N}\alpha_i}
{\prod_{i=N_1+1,\
\widetilde{\alpha}_i\not=0}^{N}\widetilde{\alpha}_i}
\frac{\prod_{i=N_1+1,\ \alpha_i=0}^{N}\gamma_i}{\prod_{i=N_1+1,\
\tilde\alpha_i=0}^{N}\gamma_i}
$$
in the case when the sets $\{\alpha_{N_1+1},\dots,\alpha_N\}$ and
$\{\tilde\alpha_{N_1+1},\dots,\tilde\alpha_N\}$ have equal numbers
of zero elements, or faces a contradiction. Having divided both
sides of~\eqref{Eq_F/tildeFb} by $\exp(a)$ and then taking the
logarithm of the result, one arrives at
\begin{multline*}
    \sum_{i=1}^{N_1}\ln\Bigl(1+\frac{\alpha_i}{\gamma_i}\frac{1}{\tau}\Bigr)+\sum_{i=N_1+1,\ \alpha_i\not=0}^{N}\ln\left(1-\frac{\gamma_i}{\alpha_i}\frac{1}{\tau}\right)-\\
    -\sum_{i=1}^{N_1}\ln\Bigl(1+\frac{\widetilde{\alpha}_i}{\gamma_i}\frac{1}{\tau}\Bigr)-\sum_{i=N_1+1,\ \widetilde{\alpha}_i\not=0}^{N}\ln\Bigl(1-\frac{\gamma_i}{\widetilde{\alpha}_i}\frac{1}{\tau}\Bigr)+
    {o}(\tau^{-M})=0.
\end{multline*}

The Taylor expansion of logarithms yields that for any natural $M$
\begin{multline}\label{Eq_Tr_2}
     -\sum_{j=1}^{M}\frac{(-1)^j}{j\tau^j}\sum_{i=1}^{N_1}\Bigl(\frac{\alpha_i}{\gamma_i}\Bigr)^j
     -\sum_{j=1}^{M}\frac{1}{j\tau^j}\sum_{i=N_1+1,\ \alpha_i\not=0}^{N}\Bigl(\frac{\gamma_i}{\alpha_i}\Bigr)^j+\\
    +\sum_{j=1}^{M}\frac{(-1)^j}{j\tau^j}\sum_{i=1}^{N_1}\Bigl(\frac{\widetilde{\alpha}_i}{\gamma_i}\Bigr)^j
    +\sum_{j=1}^{M}\frac{1}{j\tau^j}\sum_{i=N_1+1,\ \widetilde{\alpha}_i\not=0}^{N}\Bigl(\frac{\gamma_i}{\widetilde{\alpha}_i}\Bigr)^j
    +{o}(\tau^{-M})=0.
\end{multline}
Comparing coefficients at equal powers of $\tau$ now completes the
proof.
\end{proof}

We will revisit the trace formulae of Theorem
\ref{Th_Trace_Form_Com} in the next Section, see Corollary
\ref{Referee_demand}. The named Corollary reformulates necessary
conditions of isospectrality for two graph Laplacians in a much
more transparent and easier to check form. The remainder of the
present Section is devoted to a number of results which extend the
applicability of the approach developed above in two different
directions.

First, we point out that the results obtained in this Section so
far allow generalization to the case of Shr\"odinger operators on
finite compact metric graphs. These operators in the case when all
edge potentials are assumed to be summable are correctly defined
by the differential expression $-d^2/dx^2+q(x)$, where $q\in
L^1(\Gamma)$, on the same domain as the corresponding graph
Laplacians (see conditions ($\delta$) and ($\delta'$) of Section
2).

If no additional smoothness is required of edge potentials
$q_j:=q|\Delta_j$, it is only possible to obtain the first trace
formula (i.e., for $m=1$). If however edge potentials are assumed
to be $C^\infty$, the full countable set of trace formulae is
available. In the present paper we will confine ourselves to the
former case.

\begin{thm}\label{Schrodinger}
Let $\Gamma_{\delta}$ be a marked finite compact metric graph
having $N$ vertices in which $V_1,V_2,...,V_{N_1}$ are  vertices
of $\delta$ type, whereas $V_{N_1+1},V_{N_1+2},...,V_{N}$ are of
$\delta'$ type. Let $\tilde A_{\vec{\tilde\alpha}}$ and
$A_{\vec{\alpha}}$ be two Schr\"odinger operators on the graph
$\Gamma_{\delta}$ parameterized by coupling constants
$\{\alpha_{k}\}$ and $\{\widetilde{\alpha}_{k}\}$,
$k=\overline{1,N}$, respectively. Let all edge
potentials\footnote{Note, that we do not need to assume here that
the potentials are the same for operators $\tilde
A_{\vec{\widetilde{\alpha}}}$ and $A_{\vec{\alpha}}$.} $\tilde
q_i,q_i\in L_1(\Delta_i)$ for all $i=1,\dots,n.$ Let the (point)
spectra of these two operators (counting multiplicities) be equal,
$\sigma(\tilde A_{\vec{\widetilde{\alpha}}}) =
\sigma(A_{\vec{\alpha}})$. Then the numbers of zero coupling
constants at vertices of $\delta'$ type are equal and
$$
-\sum_{i=1}^{N_1}\frac{\alpha_i}{\gamma_i}+\sum_{i=N_1+1:\
\alpha_i\neq 0}^N
\frac{\gamma_i}{\alpha_i}=-\sum_{i=1}^{N_1}\frac{\tilde\alpha_i}{\gamma_i}+\sum_{i=N_1+1:\
\tilde\alpha_i\neq 0}^N \frac{\gamma_i}{\tilde\alpha_i},
$$
where $\gamma_i=\deg V_i$.

\end{thm}

The \emph{proof} of this Theorem can be obtained using WKB
asymtotics of regular Sturm-Liouville problem solutions
\cite{Sargsyan}, see \cite{Yorzh2} for more details.

\begin{cor} Under the hypotheses of Theorem \ref{Schrodinger},

(i) If $\alpha_i=\alpha$ and $\tilde\alpha_i=\tilde\alpha$ for all
graph vertices of $\delta$ type, $i=1,\dots,N_1$, whereas
$\alpha_i=-1/\alpha$ and $\tilde\alpha_i=-1/\tilde\alpha$ for all
 graph vertices of $\delta'$ type, $i=N_1+1,\dots,N$, then
$\tilde\alpha=\alpha$.

(ii) If $\alpha_i=\alpha$ and $\tilde\alpha_i=\tilde\alpha$ for
all graph vertices, $i=1,\dots,N$, then either $\tilde
\alpha=\alpha$ or $\tilde\alpha
\alpha=-(\sum_{i=N_1+1}^N\gamma_i)/(\sum_{i=1}^{N_1}\gamma_i^{-1})$,
where $\gamma_i=\deg V_i$.
\end{cor}

We finish this Section with a complement to Theorem
\ref{Th_Trace_Form_Com} which prohibits the ``decoupled''
Laplacian $A_\infty$ to have spectrum coinciding with that of an
operator $A_{\vec\alpha}$ for any $\vec\alpha$ provided that all
graph vertices are of $\delta$ type together with an analogous
result for the case of $\delta'$ vertices. In order to prove this,
one can use the following result.

\begin{prop}[\cite{DM}]\label{Prop_BoundTriple}
    Let $A_{\mathrm{min}}$ be a closed densely defined symmetric operator with equal deficiency indices,
     $(\mathcal{H},\ \Gamma_0,\ \Gamma_1)$ be a boundary triple of the operator $A_{\mathrm{max}}$,
     and finally let $M(\lambda)$ be the corresponding Weyl-Titchmarsh function.
    Then for any bounded in $\mathcal H$ self-adjoint operator $K$
\begin{enumerate}
    \item $(\mathcal{H},\ \Gamma_0,\ \Gamma_1+K\Gamma_0)$ is also a boundary triple for the operator  $A_{\max}$.
    Moreover, the corresponding Weyl-Titchmarsh function admits the form $\widehat{M}(\lambda)= M(\lambda)+K;$
    \item If the operator $A_B$ is an almost solvable w.r.t. the boundary triple $(\mathcal{H},\ \Gamma_0,\ \Gamma_1)$
    extension of the operator $A_{\mathrm{min}}$, this operator is almost solvable w.r.t. the boundary triple
    $(\mathcal{H},\ \Gamma_0,\ \Gamma_1+K\Gamma_0)$ as well. Its parameterization
    is then
    $\widehat{B}=B+K$.
\end{enumerate}
\end{prop}
This allows to prove the following
\begin{thm}\label{Prop_Isosp_A_inf_Delta}
    Let $\Gamma_{\delta}$ be a finite compact metric graph with $N$ vertices.
    Let $A_{\vec\alpha}$ be a graph Laplacian defined on  $\Gamma$
    with matching conditions of $\delta$ type at all vertices. Then the spectra of the operators $A_{\vec\alpha}$
    and
    $A_{\infty}$ coincide for no non-zero parameterizing matrix $B=\mathrm{diag}\{\alpha_{1},\alpha_{2},\dots,\alpha_{N}\}$.
\end{thm}

\begin{proof}
First assume that the matrix $B$ is invertible. In this case, pass
over to the boundary triple $(\mathcal{H},\
\widehat{\Gamma}_0=\Gamma_1,\ \widehat{\Gamma}_1=-\Gamma_0)$. The
operator $A_B$ is clearly an almost solvable extension of
$A_{\min}$ w.r.t. this triple, i.e.,  $A_B=A_{\widehat{B}}$, where
$\widehat{B}=-B^{-1}$. The Weyl-Titchmarsh function admits the
form $\widehat{M}(\lambda)=-M(\lambda)^{-1}$, and the operator
$A_{\infty}$ in this ``new'' boundary triple corresponds to $A_0$.
Hence in terms of the amended boundary triple one has
isospectrality for the operators $A_{B^{-1}}$ and $A_0$.

In the case of all vertices having $\delta$ type one clearly has
the following asymptotics for $M(\lambda)$: $
M(\lambda)=-\tau\Gamma_N+o(\tau^{-M})\quad \forall M>0 $ and hence
$$
\widehat{M}(\lambda)=(1/\tau)\Gamma_N^{-1}+o(\tau^{-M})\quad
\forall M>0,
$$
where $\Gamma_N=\mathrm{diag }\{\gamma_1,\dots,\gamma_N\}$.

As shown in the proof of Theorem~\ref{Th_Trace_Form_Com}, this
asymptotic expansion together with the condition of isospectrality
yields that both matrices parameterizing extensions must have
equal ranks, which leads to a contradiction.

Now let the matrix
$B=\mathrm{diag}\{\alpha_1,\alpha_2,...,\alpha_N\}$ degenerate.
Consider the new boundary triple defined as follows:
$(\mathcal{H},\ \widehat{\Gamma}_0=\Gamma_0,\
\widehat{\Gamma}_1=\Gamma_1+\alpha \Gamma_0)$, where
$\alpha>\mathrm{max}{|\alpha_i|,\ i=\overline{1,N}}$.
Proposition~\ref{Prop_BoundTriple} gives $A_B=A_{\widehat{B}}$,
where
$\widehat{B}=\mathrm{diag}\{\alpha_1+\alpha,\alpha_2+\alpha,...,\alpha_N+\alpha\}$
is a non-degenerate matrix, whereas the operator $A_{\infty}$
remains the same. The Weyl-Titchmarsh function admits the form
$\widehat{M}(\lambda)=M(\lambda)+\alpha I$. We have thus reduced
the situation to the one already considered.
\end{proof}

An analogous result holds with a similar proof in the case when
all graph vertices are of $\delta'$ type.
\begin{thm}\label{Prop_Isosp_A_inf_Delta'}
    Let $\Gamma_{\delta}$ be a finite compact metric graph with $N$ vertices.
    Let $A_{\vec\alpha}$ be a graph Laplacian defined on  $\Gamma$
    with matching conditions of $\delta'$ type at all vertices. Then the spectra of the operators $A_{\vec\alpha}$
    and
    $A_{\infty}$ coincide for no non-zero parameterizing matrix $B=\mathrm{diag}\{\alpha_{1},\alpha_{2},\dots,\alpha_{N}\}$.
\end{thm}

\section{Isospectral graphs by examples}

\subsection{Further dissemination of trace formulae}
We start with the following
\begin{lem}\label{dissem}
Every solution of the infinite system of equations
\begin{equation}\label{beta}
 \sum_{i=1}^N
\beta_i^m=\sum_{i=1}^N \tilde\beta_i^m\quad m=1,2,\dots.
\end{equation}
is one of the following:
\begin{equation}\label{beta_solution}
\beta_i=\tilde\beta_{k_i}\quad i=1,2,\dots,N,
\end{equation}
where $\{k_1,k_2,\dots,k_N\}$ is a permutation of the finite
sequence $\{1,2,\dots,N\}$. Conversely, each of
\eqref{beta_solution} solves the system \eqref{beta}.
\end{lem}
\begin{proof}
First assume that all $|\beta_i|$ are different. Let $\beta_{i_0}$
be the largest (by modulus) element of the sequence
$\{\beta_1,\beta_2,\dots,\beta_N\}$. Divide both parts of
\eqref{beta} by $\beta_{i_0}^m$ and pass to the limit as
$m\to\infty$. Then $\beta_{i_0}=\tilde\beta_{j_0}$ for some $j_0$,
where $\tilde\beta_{j_0}$ is the largest (by modulus) element of
the sequence
$\{\tilde\beta_1,\tilde\beta_2,\dots,\tilde\beta_N\}$.
\eqref{beta} then takes the form
$$
 \sum_{i=1, i\not=i_0}^N
\beta_i^m=\sum_{j=1, j\not=j_0}^N \tilde\beta_j^m\quad
m=1,2,\dots.
$$

Repeating the argument $N-2$ times, one arrives at the claimed
result. The generalization to the case of repeating $|\beta_i|$ is
trivial.
\end{proof}

In particular together with Theorem \ref{Th_Trace_Form_Com} this
yields that not only the numbers of zero coupling constants at
$\delta'$ type vertices have to be equal in $\vec\alpha$ and
$\vec{\tilde\alpha}$, but also the \emph{total} numbers of zero
coupling constants have to be the same under the condition of
isospectrality.

Hence, Theorem \ref{Th_Trace_Form_Com}, although leaving
possibilities for isospectrality of graph Laplacians,
significantly narrows down the set of ``opportunities'' for this
isospectrality. Precisely, one has the following
\begin{cor}\label{Referee_demand}
Let $\Gamma_{\delta}$ be a finite compact metric graph with $N$
vertices. Let $A_{\vec\alpha},\ A_{\vec{\widetilde{\alpha}}}$ be
two graph Laplacians on $\Gamma_{\delta}$ parameterized by
coupling constants $\{\alpha_{k}\}$ and
$\{\widetilde{\alpha}_{k}\}$, $k=\overline{1,N}$, respectively.
Then the graph Laplacians  $A_{\vec\alpha}$ and
$A_{\vec{\widetilde{\alpha}}}$ are isospectral only if the sets
$S$ and $\tilde S$ are permutations of each other, where
$S=\{\sigma_i\}_{i=1}^N$ with $\sigma_j=-\alpha_j/\gamma_j$ for
$V_j$ of $\delta$ type; $\sigma_j=\gamma_j/\alpha_j$ for $V_j$ of
$\delta'$ type with non-zero coupling constant; $\sigma_j=0$ in
the remaining case. The set $\tilde S$ is defined in absolutely
the same way based on coupling constants
$\{\widetilde{\alpha}_{k}\}$.
\end{cor}

Using the latter Corollary, for any given graph $\Gamma$ one
could, at least theoretically, try every ``allowed'' configuration
of coupling constants one by one and thus assess directly, which
of these (if any) lead to isospectral configurations.

\begin{eg}
Let $\Gamma_{\delta}$ be a decorated ``lasso'' graph of three
vertices all of them being of $\delta$ type.
\begin{center}
\includegraphics[width=.95\textwidth]{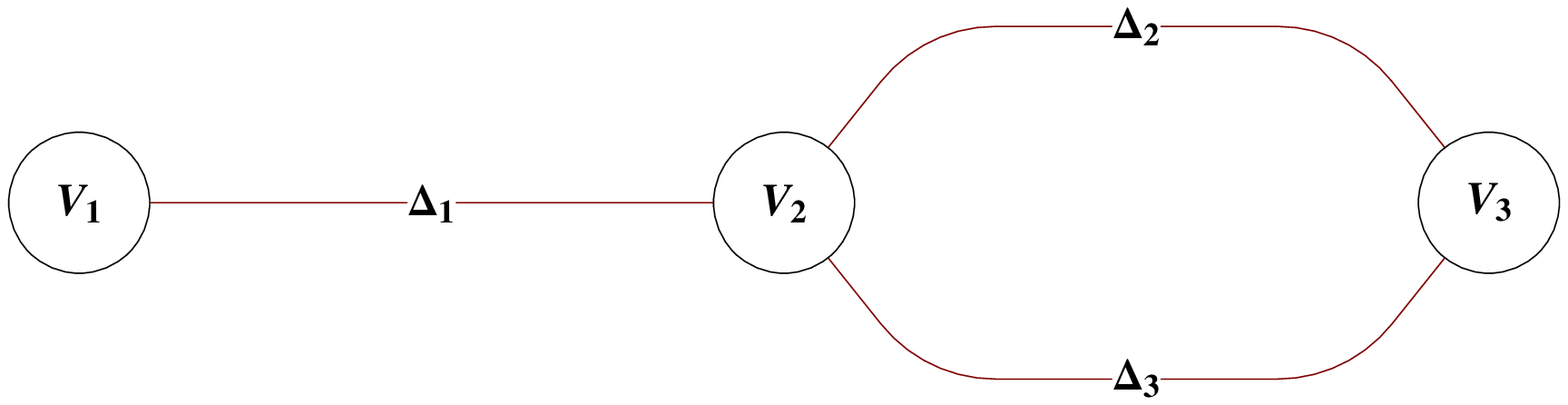}
\end{center}
Let all three edges of the graph be of length 1. Let
$\vec{\alpha}=(a,3b,2c)$ for arbitrary real $a,b,c\in\mathbb{R}$.
Then $A_{\vec{\widetilde{\alpha}}}$ is isospectral to
$A_{\vec{\alpha}}$ for $\vec{\widetilde{\alpha}}=(b,3c,2a)$
provided, that $2b=a+c$ and to
$\vec{\widetilde{\alpha}}=(c,3a,2b)$ provided, that $2a=b+c$.

Note, that both configurations $(b,3c,2a)$ and $(c,3a,2b)$ are
among allowed by Corollary \ref{Referee_demand}.

At the same time, if the lengths of graph edges are rationally
independent, there are no isospectral configurations.
\end{eg}

The \emph{proof} is an explicit computation.

This example demonstrates that if one wishes to go one step
further than Theorem \ref{Th_Trace_Form_Com} in the analysis of
isospectrality for graph Laplacians, one needs to impose some
additional restrictions on edge lengths. The analysis presented in
Section 4, based on the asymptotic behavior of Weyl-Titchmarsh
M-function, shows that the information derived from the leading
term of its asymptotics carries no information on edge lengths.
This asks for some additional considerations, to which the
remainder of the present Section is devoted.

\emph{Henceforth we will assume that all edge lengths are
rationally independent} and consider a number of examples of
quantum graphs.

\subsection{The case of a star-graph}

\begin{thm}\label{Star}
Consider a finite compact metric graph $\Gamma_{\delta}$ which is
a star, i.e., $\Gamma_{\delta}$ has exactly $N+1$, $N\geq 2$
vertices\footnote{If $N=2$, the coupling constant at the internal
vertex is w.l.o.g. assumed to be non-zero, since otherwise the
star-graph considered reduces to a single interval by graph
cleaning, see, e.g., \cite{Aharonov}.} with the only internal
vertex being $V_{N+1}$. Assume that all edge lengths are
rationally independent and that matching conditions at all
vertices are of $\delta$ type. Let $A_{\vec\alpha}$ and
$A_{\vec{\widetilde{\alpha}}}$ be two graph Laplacians on
$\Gamma_{\delta}$. Assume that their spectra coincide counting
multiplicities. Then $\vec\alpha=\vec{\widetilde{\alpha}}$.
\end{thm}

\begin{proof}
As in the proof of Theorem \ref{Th_Trace_Form_Com}, one
immediately ascertains the equality
\begin{equation}\label{Eq_Dets_8}
\mathrm{det}(M(\lambda)-B)=\mathrm{det}(M(\lambda)-\widetilde{B}),
\end{equation}
since in this case it is easy to see that $\exp(a)=1$ in
\eqref{Eq_F/tildeFb}. By Theorem \ref{Th_Weyl_Func},
\begin{multline*}
\mathrm{det}(M(\lambda)-B)=\\ \left|
  \begin{array}{cccc}
    -\mu\cot\mu l_1-\alpha_1 & 0 & \ldots &\mu\csc\mu l_1 \\
    0 & -\mu\cot\mu l_2-\alpha_2 & \ldots &\mu\csc\mu l_2 \\
    \vdots & \vdots & \ldots &  \vdots \\
    \mu\csc\mu l_1 & \mu\csc\mu l_2 & \ldots & -\mu\sum_{t=1}^{N}\cot\mu l_t-\alpha_{N+1} \\
  \end{array}
\right|,
\end{multline*}
where $\mu=\sqrt{\lambda}$. Calculating this determinant
explicitly, one has
$$
\det(M(\lambda)-B)=K_{N+1}(\mu)\mu^{N+1}+K_N(\mu,\vec{\alpha})\mu^N+\dots
+K_0(\mu,\vec{\alpha}),
$$
where $K_j(\mu)$ are trigonometric functions. Here
$\mu^{N+1}K_{N+1}(\mu)=\det M(\lambda)$ and (up to the sign)
\begin{equation}\label{Eq_5}
K_{N}(\mu)= \sum_{t=1}^{N+1}\alpha_t\prod_{j=1}^{N}\cot\mu
l_j-\sum_{t,i=1;\ t<i}^{N} (\alpha_t+\alpha_i)\prod_{j\neq
t,i}\cot\mu l_j,
\end{equation}
Utilizing the standard linear independence argument in conjunction
with the condition of rational independence of edge lengths, one
obtains
\begin{gather}\label{Eq_K_n_1}
\sum_{t=1}^{N+1}\alpha_t=\sum_{t=1}^{N+1}\widetilde{\alpha}_t;\\
\label{Eq_K_n_2}
\alpha_t+\alpha_i=\widetilde{\alpha}_t+\widetilde{\alpha}_i,\
t,i=\overline{1,N},\ t<i.
\end{gather}
Using the first trace formula of Theorem \ref{Th_Weyl_Func}, one
additionally has
$$
\sum_{t=1}^{N}\alpha_t+\frac{\alpha_{N+1}}{N}=\sum_{t=1}^{N}\widetilde{\alpha}_t+\frac{\widetilde{\alpha}_{N+1}}{N},
$$
which together with~\eqref{Eq_K_n_1} implies $
\alpha_{N+1}=\widetilde{\alpha}_{N+1}. $ Then~\eqref{Eq_K_n_1}
admits the form
\begin{equation}\label{Eq_K_n_3}
\sum_{t=1}^{N}\alpha_t=\sum_{t=1}^{N}\widetilde{\alpha}_t.
\end{equation}
On the other hand, \eqref{Eq_K_n_2} immediately yields
\begin{equation}\label{Eq_K_n_4}
(N-1)\alpha_{1}+\sum_{t=2}^{N}\alpha_t=(N-1)\widetilde{\alpha}_{1}+\sum_{t=2}^{N}\widetilde{\alpha}_t.
\end{equation}
Comparing \eqref{Eq_K_n_4} with~\eqref{Eq_K_n_1}, one obtains
$$
\alpha_{1}=\widetilde{\alpha}_{1}
$$
and, proceeding analogously, the same result for all
$i=\overline{2,N}$.
\end{proof}

\subsection{The case of a chain graph}
We will demonstrate that unless the chain is also a star, it no
longer siffices to consider linear relations between coupling
constants in order to ascertain the fact that there are no
isospectral graphs in this situation. In the present paper, we
will only consider the chain of exactly 4 vertices, which is
already enough to illustrate the point made. The general case of
an arbitrary compact chain can in fact be reduced to this one by a
corresponding recurrence relation.

\begin{prop}\label{A4}
Let $\Gamma_{\delta}$ be the metric $A_4$ graph with $\delta$ type
matching conditions at all vertices  and having rationally
independent edge lengths. Let $A_{\vec\alpha}$ and
$A_{\vec{\widetilde{\alpha}}}$ be two graph Laplacians on
$\Gamma_{\delta}$. Let additionally coupling constants at both
internal vertices be non-zero for both Laplacians.\footnote{Again,
this assumption just guarantees that the graph $\Gamma$ does not
reduce by the procedure of cleaning \cite{Aharonov} to either a
star or an interval.} Assume that their spectra coincide counting
multiplicities. Then $\vec\alpha=\vec{\widetilde{\alpha}}$.
\end{prop}

\begin{proof} As in the proof of Theorem \ref{Star},
\eqref{Eq_Dets_8} holds. By Theorem \ref{Th_Weyl_Func}, one has
${\det}(M(\lambda)-B)=$
\begin{tiny}
\begin{equation*}
\left|
  \begin{array}{cccc}
    -\mu\cot\mu l_1-\alpha_1 & \mu\csc\mu l_1  & 0 &0 \\
    \mu\csc\mu l_1 & -\mu\sum_{n=1}^{2}\cot\mu l_t-\alpha_2 &\mu\csc\mu l_2 & 0 \\
    0 & \mu\csc\mu l_2 & -\mu\sum_{n=2}^{3}\cot\mu l_t-\alpha_3 &  \mu\csc\mu l_3 \\
    0 & 0& \mu\csc\mu l_3 & -\mu\cot\mu l_3-\alpha_4 \\
  \end{array}
\right|,
\end{equation*}
\end{tiny}where the enumeration of vertices is chosen so that that the
internal vertices are $V_2$ and $V_3$, the edge between them being
of length $l_2$, whereas the edge of length $l_1$ connects $V_2$
to $V_1$ and consequently $l_3$ connects $V_3$ and $V_4$.

Proceeding exactly as in the proof of Theorem \ref{Star}, one gets
the following linear relations between $\vec{\alpha}$ and
$\vec{\widetilde{\alpha}}$ from consideration of the coefficient
$K_3(\mu,\vec{\alpha})$:
\begin{align*}
  \alpha_1+\alpha_2+\alpha_3+\alpha_4&=\widetilde{\alpha}_1+\widetilde{\alpha}_2+\widetilde{\alpha}_3+\widetilde{\alpha}_4;\\
  \alpha_1+\alpha_2+\alpha_4&=\widetilde{\alpha}_1+\widetilde{\alpha}_2+\widetilde{\alpha}_4;\\
  \alpha_1+\alpha_4&=\widetilde{\alpha}_1+\widetilde{\alpha}_4;\\
  \alpha_1+\alpha_3+\alpha_4&=\widetilde{\alpha}_1+\widetilde{\alpha}_3+\widetilde{\alpha}_4.
\end{align*}
The matrix of this system of linear equations has rank equal to 3
with kernel $(1,0,0,-1)^T$. This of course immediately implies
$\alpha_2=\widetilde{\alpha}_2$, $\alpha_3=\widetilde{\alpha}_3,$
but still does not prove the claim. Note that the first trace
formula of Theorem \ref{Th_Trace_Form_Com} in fact follows from
the relations above and thus provides no additional information.

Consider the coefficient $K_2(\mu,\vec{\alpha})$. This equips us
with 4 quadratic relations on the coupling constants which
together with the linear relations yield the claim.
\end{proof}

\subsection{The case of mixed types with double edges}
\begin{prop}
 Let
$\Gamma_{\delta}$ be the graph of Example
\ref{Example_mexican_dupa}. Let the edge lengths $l_j$,
$j=1,\dots,4$ be rationally independent. Let $A_{\vec{\alpha}}$
and $A_{\vec{\widetilde{\alpha}}}$ be two graph Laplacians on
$\Gamma_{\delta}$. Assume that their spectra coincide counting
multiplicities. Then $\vec\alpha=\vec{\widetilde{\alpha}}$.
\end{prop}

\begin{proof}
Assume first that $\alpha_3,\alpha_4,
\tilde\alpha_3,\tilde\alpha_4\not=0$. As in the proof of Theorem
\ref{Star}, one has
$$
\frac{1}{\alpha_3\alpha_4}\det(M(\lambda)-\mathrm{diag}\{\alpha_1,\dots,\alpha_4\})
=\frac{1}{\tilde\alpha_3\tilde\alpha_4}\det(M(\lambda)-\mathrm{diag}\{\tilde\alpha_1,\dots,\tilde\alpha_4\}).
$$
By Theorem \ref{Th_Weyl_Func}, one gets the following
decomposition of
$\det(M(\lambda)-\mathrm{diag}\{\alpha_1,\dots,\alpha_4\})$:
$$
K_{2}(\mu)\mu^{2}+K_{1}(\mu,\vec{\alpha})\mu+K_{0}(\mu,\vec{\alpha})+K_{-1}(\mu,\vec{\alpha})\mu^{-1}+K_{-2}(\mu)\mu^{-2}
$$
with trigonometric coefficients $K_j$. The only difference with
Theorem \ref{Star} is that here negative powers of $\mu$ appear.
What's more, explicit calculation shows that both $K_1$ and
$K_{-1}$ already are non-linear w.r.t. the coupling constants.
Nevertheless, one ascertains almost immediately that these suffice
to complete the proof. Consideration of the remaining cases
differs only in minute details.
\end{proof}

We remark that the same result holds for the graph
$\Gamma_{\delta}$ if all vertices are assumed to be of the same
type (either $\delta$ or $\delta'$). Thus, a graph possessing
double edges may prevent isospectrality.

\subsection{Isospectrality can happen if even cycles are allowed}

We will show that a cycle $C_4$ of four vertices even in the
situation of rationally independent edge lengths allows for
existence of isospectral Laplacians.

\begin{prop}\label{C4}
Let $\Gamma_{\delta}=C_4$ with all four vertices of $\delta$ type
and rationally independent edge lengths $l_j,$ $j=1,\dots,4$.
Assume that $A_{\vec\alpha}$ and $A_{\vec{\widetilde{\alpha}}}$
are two graph Laplacians on $\Gamma_{\delta}$. Let additionally
coupling constants at all vertices be non-zero for both
Laplacians.\footnote{This assumption guarantees that the graph
$\Gamma$ does not reduce by the procedure of cleansing
\cite{Aharonov} to a cycle of lower dimension.} Assume that their
spectra coincide counting multiplicities. Then
$\vec\alpha=\vec{\widetilde{\alpha}}$, unless
$\vec{\alpha}=(\alpha,-\alpha,\alpha,-\alpha)$ and
$\vec{\widetilde{\alpha}}=(-\alpha,\alpha,-\alpha,\alpha)$ for
some real $\alpha$. In the latter case $A_{\vec\alpha}$ and
$A_{\vec{\widetilde{\alpha}}}$ are isospectral.
\end{prop}

\begin{proof}
As in the proof of Theorem \ref{Star}, \eqref{Eq_Dets_8} holds.
Calculating this determinant explicitly on the basis of Theorem
\ref{Th_Weyl_Func}, one has
$$
\det(M(\lambda)-B)=K_{4}(\mu)\mu^{4}+K_3(\mu,\vec{\alpha})\mu^3+\dots
+K_0(\mu,\vec{\alpha}),
$$
where $K_j(\mu)$ are trigonometric functions. Calculating these
functions explicitly and utilizing the standard linear
independence argument in conjunction with the condition of
rational independence of edge lengths, one obtains
\begin{equation}\label{Eq_C_1}
\alpha_1+\alpha_2+\alpha_3+\alpha_4=\widetilde{\alpha}_1+\widetilde{\alpha}_2+\widetilde{\alpha}_3+\widetilde{\alpha}_4;
\end{equation}
\begin{equation}\label{Eq_C_2}
\begin{split}
&\alpha_j\sum_{i\not=j}\alpha_i=\widetilde{\alpha}_j\sum_{i\not=j}\widetilde{\alpha}_i,\ i,j=\overline{1,4},\\
&\alpha_1\alpha_3+\alpha_1\alpha_4+\alpha_2\alpha_3+\alpha_2\alpha_4=
\widetilde{\alpha}_1\widetilde{\alpha}_3+\widetilde{\alpha}_1\widetilde{\alpha}_4+\widetilde{\alpha}_2\widetilde{\alpha}_3+
\widetilde{\alpha}_2\widetilde{\alpha}_4,\\
&\alpha_1\alpha_2+\alpha_1\alpha_3+\alpha_2\alpha_4+\alpha_3\alpha_4=
\widetilde{\alpha}_1\widetilde{\alpha}_2+\widetilde{\alpha}_1\widetilde{\alpha}_3+\widetilde{\alpha}_2\widetilde{\alpha}_4+
\widetilde{\alpha}_3\widetilde{\alpha}_4,\\
&\alpha_1\alpha_2+\alpha_1\alpha_4+\alpha_2\alpha_3+\alpha_3\alpha_4=
\widetilde{\alpha}_1\widetilde{\alpha}_2+\widetilde{\alpha}_1\widetilde{\alpha}_4+\widetilde{\alpha}_2\widetilde{\alpha}_3+
\widetilde{\alpha}_3\widetilde{\alpha}_4;
\end{split}
\end{equation}
\begin{equation}\label{Eq_C_3}
 \begin{split}
 &\alpha_1\alpha_3\alpha_4+\alpha_2\alpha_3\alpha_4=\widetilde{\alpha}_1\widetilde{\alpha}_3\widetilde{\alpha}_4+
 \widetilde{\alpha}_2\widetilde{\alpha}_3\widetilde{\alpha}_4,\\
 &\alpha_1\alpha_2\alpha_4+\alpha_1\alpha_3\alpha_4=\widetilde{\alpha}_1\widetilde{\alpha}_2\widetilde{\alpha}_4+
 \widetilde{\alpha}_1\widetilde{\alpha}_3\widetilde{\alpha}_4,\\
 &\alpha_1\alpha_2\alpha_3+\alpha_1\alpha_2\alpha_4=\widetilde{\alpha}_1\widetilde{\alpha}_2\widetilde{\alpha}_3+
 \widetilde{\alpha}_1\widetilde{\alpha}_2\widetilde{\alpha}_4,\\
 &\alpha_1\alpha_2\alpha_3+\alpha_2\alpha_3\alpha_4=\widetilde{\alpha}_1\widetilde{\alpha}_2\widetilde{\alpha}_3+
 \widetilde{\alpha}_2\widetilde{\alpha}_3\widetilde{\alpha}_4;
 \end{split}
\end{equation}
\begin{equation}\label{Eq_C_4}
  \alpha_1\alpha_2\alpha_3\alpha_4=\widetilde{\alpha}_1\widetilde{\alpha}_2\widetilde{\alpha}_3\widetilde{\alpha}_4.
\end{equation}
In variables $x_{i,j}:=\alpha_i
\alpha_j-\widetilde{\alpha}_i\widetilde{\alpha}_j$, $i,j=1,4,\
i<j$, the system \eqref{Eq_C_2} turns out to be a linear system
with matrix of rank 6. Hence,
\begin{equation}\label{Eq_Second}
\alpha_i\alpha_j=\widetilde{\alpha}_i\widetilde{\alpha}_j,\
i,j=\overline{1,4}.
\end{equation}

As for~\eqref{Eq_C_3}, one reduces this system to a linear one by
the following change of variables:
\begin{equation}\label{Eq_Terc}
\begin{split}
&x_1=\alpha_1\alpha_2\alpha_3-\widetilde{\alpha}_1\widetilde{\alpha}_2\widetilde{\alpha}_3,\
x_2=\alpha_1\alpha_2\alpha_4-\widetilde{\alpha}_1\widetilde{\alpha}_2\widetilde{\alpha}_4,\\
&x_3=\alpha_1\alpha_3\alpha_4-\widetilde{\alpha}_1\widetilde{\alpha}_3\widetilde{\alpha}_4,\
x_4=\alpha_2\alpha_3\alpha_4-\widetilde{\alpha}_2\widetilde{\alpha}_3\widetilde{\alpha}_4.
\end{split}
\end{equation}
Since the rank of the corresponding matrix is equal to 3, the
general solution of the latter system is non-trivial. Further, it
has the form $(\beta,-\beta,\beta,-\beta).$ Consider the two
possibilities separately.

(\textbf{i}) If $\beta=0$, then~\eqref{Eq_C_4} immediately implies
$\alpha_i-\widetilde{\alpha}_i=0,\ j=\overline{1,4}$.

(\textbf{ii}) Let now $\beta\neq 0$. Denoting
$u:=\alpha_2\alpha_3,\ v:=\alpha_2\alpha_4,\ w:=\alpha_3\alpha_4$
and using~\eqref{Eq_Terc} and~\eqref{Eq_Second} one has:
$$
(\alpha_1-\widetilde{\alpha}_1)u=\beta,\
(\alpha_1-\widetilde{\alpha}_1)v=-\beta,\
(\alpha_1-\widetilde{\alpha}_1)w=\beta,
$$
whence $u=-v=w,$ or equivalently
$\alpha_2\alpha_3=-\alpha_2\alpha_4=\alpha_3\alpha_4$. Thus,
$$
\alpha_3=-\alpha_4,\  -\alpha_2=\alpha_3.
$$
In the same way one might also ascertain that $\alpha_1=\alpha_3.$
Finally one gets $\alpha_1=\alpha_3=\alpha,\
\alpha_2=\alpha_4=-\alpha$,
$\widetilde{\alpha}_1=\widetilde{\alpha}_3=\widetilde{\alpha},\
\widetilde{\alpha}_2=\widetilde{\alpha}_4=-\widetilde{\alpha}.$

From~\eqref{Eq_C_4} it follows that
$|\alpha|=|\widetilde{\alpha}|.$ If
$\alpha\neq\widetilde{\alpha},$ the only remaining possibility of
isospectral graph Laplacian turns out to be
$\widetilde{\alpha}=-\alpha$.

An explicit calculation shows that in this case one indeed gets
the required property \eqref{Eq_Dets_8}.
\end{proof}

Some remarks are in order.

1. Proposition \ref{C4} admits generalization to an arbitrary
cycle $C_{2N}$, $M>1$. As for the cycle $C_2$, technically one
could prove isospectrality if and only if
$\vec{\alpha}=(\alpha,-\alpha)$ and
$\vec{\widetilde{\alpha}}=-\vec{\alpha}$, however, this does not
lead to non-trivial isospectral configurations since the
corresponding two graphs are actually identical.

2. As for the case of cycles $C_{2N+1}$, $N\geq 1$, one can
ascertain much along the same lines that such odd cycles do not
permit isospectrality.

3. In our forthcoming publication on necessary and sufficient
conditions of isospectrality, we will show that essentially there
are no other examples of isospectral graph Laplacians provided
that all graph vertices are of $\delta$ type.

\subsection*{Acknowledgements} The authors express deep gratitude
to Profs. Sergey Naboko and Yurii Samoilenko for constant
attention to their work. We would also like to cordially thank our
referees for making some very helpful comments.

\end{document}